\documentclass[a4paper,10pt,leqno,english]{article}
\usepackage[T1]{fontenc}
\usepackage[utf8]{inputenc}
\usepackage[a4paper, margin=2cm]{geometry}
\usepackage{babel}
\usepackage{amsthm,amsmath,amsfonts,amssymb,titlesec}
\usepackage{enumitem,pifont,physics,comment}
\usepackage{hyperref}
\newtheorem{theorem}{Theorem}[section]

\newtheorem{lemma}{Lemma}[section]
\makeatletter
\renewcommand\theequation%
{\thesection.\arabic{equation}}
\@addtoreset{equation}{section}
\makeatother
\newtheorem{remark}{Remark}[section]

\usepackage[
backend=biber,
style=alphabetic,
sorting=nyt
]{biblatex}
\usepackage{csquotes}
\usepackage{amsmath,amsthm,amssymb,amscd,color, xcolor,mathtools,url,tikz}
\usepackage{tikz}
\usetikzlibrary{positioning}

\def\R{\mathbb{R}}
\def\N{\mathbb{N}}
\def\Z{\mathbb{Z}}

\let\eps\varepsilon
\let\epsilon\varepsilon

\let\oldsum\sum
\renewcommand{\sum}{\displaystyle\oldsum}
\let\oldprod\prod
\renewcommand{\prod}{\displaystyle\oldprod}
\let\oldinf\inf
\renewcommand{\inf}{\displaystyle\oldinf}
\let\oldsup\sup
\renewcommand{\sup}{\displaystyle\oldsup}

\let\leq\leqslant
\let\geq\geqslant

\newlength{\oldparindent}
\newcommand{\myindent}{\hspace{\oldparindent}}

\title{Weakly turbulent solution to the Schrödinger equation on the two-dimensional torus with real potential decaying to zero at infinity}
\author{Ambre Chabert}
\date{}

\begin{document}
\maketitle

\begin{abstract}
    We build a smooth time-dependent real potential on the two-dimensional torus, decaying as time tends to infinity in Sobolev norms along with all its time derivatives, and we exhibit a smooth solution to the associated Schrödinger equation on the two-dimensional torus whose $H^s$ norms nevertheless grow logarithmically as time tends to infinity. We use Fourier decomposition in order to exhibit a discrete resonant system of interactions, which we are further able to reduce to a sequence of finite-dimensional linear systems along which the energy propagates to higher and higher frequencies. The constructions are very explicit and we can thus obtain lower bounds on the growth rate of the solution. 
\end{abstract}

\section{Introduction}

\subsection{Main result}

\myindent In this paper, we build an explicit $\mathcal{C}^{\infty}$ solution to the Schrödinger equation on the two-dimensional torus $\mathbb{T}^2 := \R^2/(2\pi\Z)^2$
\begin{equation}\label{main}
    i\partial_t u(t,x) = -\Delta u (t,x) + V(t,x)u (t,x) \quad (t,x) \in [0,+\infty) \times \mathbb{T}^2
\end{equation}
where the potential $V(t,x)$ is real, smooth on the interval $[0, + \infty)\times \mathbb{T}^2$, and decaying at infinity in Sobolev norms.

\myindent With a carefully chosen $V$, we are able to exhibit \textit{weakly turbulent} behaviour, that is we are able to prove the following theorem

\begin{theorem}\label{result} There exist a real smooth potential $V(t,x)$, and a smooth function $u(t,x)$, $(t,x)\in [0,+\infty)\times \mathbb{T}^2$ such that 
    \begin{equation}
        i\partial_t u(t,x) = -\Delta u(t,x) + V(t,x)u(t,x)
    \end{equation}
    \myindent Furthermore, given any small constant $\delta > 0$, and any order $s > 0$, there exists  $c_{\delta,s} > 0$ such that as $t\to \infty$
    \begin{equation}\label{grate}
        \|u(t)\|_{H^s} \geq c_{\delta,s} (\log t)^{s(1 - \delta)}.
    \end{equation}
    \myindent Finally, the potential $V$ satisfies the bound
    \begin{equation}\label{decay}
        \forall k \in \N, \ \forall s\geq 0 \quad \lim_{t\to \infty} \|\partial_t^kV(t,\cdot)\|_{H^s} = 0
    \end{equation}
\end{theorem}

\myindent We will in the last section explore possible upper bounds for the decay rate of $V$, which is subpolynomial, see \eqref{decayV}.

\subsection{Earlier work}

\myindent The first exemple of unbounded growth of the Sobolev norms for the Schrödinger equation \eqref{main} on the torus $\mathbb{T}^2$ was given by Bourgain in \cite{bourgain1999growtha}, although the potential $V$ is chosen to be \textit{quasiperiodic}. Bourgain proves that a logarithmic growth of the Sobolev norms can be achieved in this setting, and that is is optimal. Bourgain also studied the case of a random behaviour in time with certain smoothness conditions in \cite{bourgain1999growthb}. Furthermore, Bourgain proves in those articles that with a bounded smooth potential $V$ then the growth in any norm $H^s$ is bounded by $t^{\eps}$ for all $\eps > 0$ (with a constant that depends upon $s,V,\eps$) ; and that for a potential \textit{analytic} in time the bound can be refined to $(\log(t))^{\alpha}$.

\myindent With regards to the logarithmic growth rate we are able to achieve in the present article, it is necessarily subpolynomial as $V$ is assumed to be smooth and bounded, but we may not use the logarithmic a priori bound as $V(t)$ is not analytic in $t$ in our construction. Still, logarithmic growth rate is nearly optimal as the optimal growth is necessarily subpolynomial.

\myindent The study of upper bounds on the possible growth rate of Sobolev norms of the solutions to linear Schrödinger equation has a long history. The general question can be formulated as follow : consider $u$ a regular solution to 

\begin{equation}\label{linearperteq}
    i\partial_t u = H u + P(t)u
\end{equation}

where $H$ is either the Laplacian $-\Delta$ on a $d$-dimensional torus, either more generally, when the domain is $\R^d$ or even a manifold, a time-independent self-adjoint nonnegative operator with some assumptions on its spectrum, and $P(t)$ is a smooth time-dependent family of pseudo-differential operators of order strictly lower than 2. Then one can try and prove upper bound on the growth rate of $\|u(t)\|_{H^s}$ as $t\to \infty$.

\myindent \cite{maspero2017time} proved, along with global well-posedness, $t^{\eps}$ upper bound on the growth rate in the case where $H$ has an increasing spectral gap (as is the case for the Laplacian on Zoll manifolds) and $P(t)$ is a smooth perturbation. This bound can be improved to $(\log(t))^{\gamma}$ for some $\gamma > 0$ when $P(t)$ is analytic in time, which is reminiscent of Bourgain's bound. Using those results, \cite{bambusi2021growth} proves $t^{\eps}$ upper bounds on the growth rate of solutions to \eqref{linearperteq} in an abstract setting, which includes in particular the case where $H$ is the Harmonic Oscillator in $\R^d$ and $P(t)$ is a pseudodifferential operator of order strictly lower than $H$ depending in a quasiperiodic way on time. The first result of a $t^{\eps}$ upper bound with an \textit{unbounded} $P(t)$ was obtained in \cite{bambusi2022growthb} on the torus $\mathbb{T}^d$ with $H = -\Delta$. Finally, $t^{\eps}$ upper bounds have been proved for general hamiltonians of quantum integrable systems in \cite{bambusi2022growtha}.

\quad

\myindent Regarding the dual question of exhibiting growth of Sobolev norms in solutions to \eqref{linearperteq}, the recent articles of Maspero \cite{maspero2022growth} and \cite{maspero2023generic} proved the existence of solutions with (unbounded) polynomial growth in the case where $H$ has a fixed spectral gap and $P(t)$ is a potential \textit{periodic} in time, using a resonance phenomenon. Loosening the time smoothness hypothesis, Erdogan, Killip and Schlag showed genericity of Sobolev norms growth when the potential is a stationary Markov process in \cite{erdougan2003energy}. See also \cite{eliasson2009reducibility}, \cite{delort2010growth},\cite{wang2008logarithmic}.

\myindent Regarding potentials whose Sobolev norms decay to zero with time more specifically, Raphael and Faou were able to exhibit logarithmic growth in the context where $H = -\Delta + |x|^2$ is the harmonic oscillator on $\R^2$ in \cite{faou2020weakly}. Their method relies on \textit{quasiconformal modulations} of so-called \textit{bubble} solutions of the unperturbed Schrödinger equation. It is not surprising that we are able to exhibit logarithmic growth on the torus as the setting is similar. Indeed, both the harmonic oscillator on $\R^2$ and the laplacian on the torus are operators with compact resolvant and a spectrum with geometric properties (as it is formed of points in a lattice) which allows for explicit resonance mechanism. Let us note that the author was able to prove in \cite{chabert2024weakly} that their method extends to the case where the cubic nonlineariy $u|u|^2$ is added to the equation, using a similar approximation scheme than in the present article.

\quad

\myindent The method we shall use here is inspired by the seminal work \cite{colliander2010transfer} refined by \cite{guardia2015growth}. Indeed, we use that on the two-dimensional torus, eigenfunctions of the Laplacian are given by $e^{in\cdot x}$ for $n \in \Z^2$, with eigenvalue $|n|^2$. The lattice structure is then used to produce resonance phenomena between carefully chosen frequencies of the Fourier decomposition of the solution $u$. The idea is that only certain \textit{resonant} interaction will dominate the behaviour of the solution, thus, using an arbitrarily small potential, we are able to transfer the energy of the solution to higher and higher frequencies, leading to growth of Sobolev norms.

\subsection{Idea of the proof}

\myindent The first step of the proof is directly inspired by \cite{colliander2010transfer}. In section 2, we decompose the equation \eqref{main} in Fourier frequencies, thus reducing it to an infinite-dimensional ODE on the Fourier frequencies $(a_n(t))$ of the solution. This enables us to exhibit some \textit{resonant interactions} between Fourier frequencies, which will dominate the behaviour of the solution in terms of Sobolev norms. In that spirit, we first study a \textit{Resonant Fourier System} where we drop the non-resonant interactions. We then build a family of Fourier frequencies $(m_n)_{n\geq 0}$, satisfying carefully computed \textit{orthogonality properties}, along which we are able to transfer energy to higher frequencies (as $|m_n| \to \infty$) with a well-tailored potential $V$ for a solution $(a_n(t))$ whose Fourier frequencies are almost supported on the $(m_n)$.

\myindent In section 3, we give a detailed construction of a potential allowing said energy transfer to higher frequencies, thanks to the crucial point that, as we only consider resonant interactions, we may light up only specific Fourier frequencies in the potential, which further reduces the Resonant System to a \textit{sequence of finite-dimensional linear systems} which we can explicitly solve.

\myindent In section 4 and 5, we prove that the solution to the Resonant System yields a solution to the full system up to a perturbation thanks to a Cauchy sequence scheme, thus controlling that the perturbation decays to zero as $t\to \infty$. We finally use the explicit construction of the solution to the Resonant System to deduce lower bounds on the growth of the Sobolev norm of the full solution, thus concluding to the proof of theorem \eqref{result}

\subsection{Acknowledgements}

\myindent I would like to express my deepest thanks to Professor Pierre Germain, who helped me greatly to format the present article through many discussions and proofreading. I would also like to thank Professor Pierre Raphaël for asking me the question solved here.

\section{Fourier decomposition and resonant system}

\subsection{Reduction to a Resonant Fourier System}

\myindent We now show how \eqref{main} can be heuristically approximated by an easier equation, focusing on the \textit{resonant} interactions. Indeed, as we wish to find smooth solutions of \eqref{main} we may write

\begin{equation}
    u(t,x) = \sum_{n\in \Z^2} a_n(t) e^{i(n\cdot x - |n|^2t)}
\end{equation}

\myindent We now set the potential to take the form

\begin{equation}
    V(t,x) = -\sum_{n\in \Z^2} 2v_n(t) \sin(|n|^2 t) e^{in\cdot x}
\end{equation}
where $v_{-n} = v_n$ is real. Thus, we need only find a solution to the $l^2$ system 

\begin{equation}\tag{$\mathcal{FS}$}
    \partial_t a_n = \sum_{m \in \Z^2} a_m(t) v_{n-m}(t) \left(e^{-i\omega_{m,n}^+t} - e^{-i\omega_{m,n}^-t}\right)
\end{equation}
where we set 

\begin{align*}
    \omega_{m,n}^+ &:= |m|^2 + |m-n|^2 - |n|^2\\
    \omega_{m,n}^- &:= |m|^2 - |m-n|^2 - |n|^2
\end{align*}

\myindent Now, in the spirit of \cite{colliander2010transfer}, we expect that the resonant interaction will dominate, that is interaction between frequencies $m,n$ such that one of $\omega_{m,n}^+$ or
$\omega_{m,n}^-$ is zero. We thus denote for $n \in \Z^2$

\begin{align*}
    \Gamma_{res}^+(n) &:= \{m\in \Z^2, \ |m|^2 + |m-n|^2 - |n|^2 = 0\}\\
    \Gamma_{res}^-(n) &:= \{m\in \Z^2, \ |m|^2 - |m-n|^2 - |n|^2 = 0\}
\end{align*}

and define the approximated system

\begin{equation}\tag{$\mathcal{RFS}$}
    \partial_t a_n = \sum_{m\in \Gamma_{res}^+(n)}a_m(t)v_{n-m}(t) \ - \sum_{m\in \Gamma_{res}^-(n)}a_m(t)v_{n-m}(t)
\end{equation}

\myindent We observe that $(\mathcal{RFS})$ conserves the $l^2$ norm. Indeed 

\begin{align*}
    \frac{d}{dt} \|(a_n)\|_{l^2}^2 = 2Re\left(\sum_{n \in \Z^2} \sum_{m \in \Gamma_{res}^+(n)} \overline{a_n(t)} a_m(t) v_{n-m}(t) - \sum_{n \in \Z^2} \sum_{m \in \Gamma_{res}^-(n)} \overline{a_n(t)} a_m(t) v_{n-m}(t)\right)
\end{align*}

\myindent However, $m \in \Gamma_{res}^+(n)$ if and only if $n \in \Gamma_{res}^-(m)$. Using moreover that $v_{-k} = v_k$ we see that the RHS equals zero.

\subsection{Geometric interpretation of the resonant frequencies}

\myindent Now, we turn our attention to the geometric interpretation of the equation $\omega_{m,n}^{+/-}$ = 0 : we first see that $\omega_{m,n}^+ = 0$ if and only if 

 \begin{equation}
    \begin{cases} m + (n-m) &= n \\
    |m|^2 + |n-m|^2 & = |n|^2 \end{cases}
    \end{equation}
which means that $m$ is orthogonal to $n-m$. This can be reformulated by saying that $m$ resonates with the $m+l$ where $l\in \mathbb{Z}^2$ is orthogonal to $m$.

\myindent Similarly we see that $\omega_{m,n}^- = 0$ if and only if $(n-m)$ is orthogonal to $n$ ; which finally means that $m$ and $n$ are resonant frequencies if one of $m$ or $n$ is the sum of the other one and of an orthogonal vector. We may sum those facts up in a lemma.

\begin{lemma}
    For all $n,m \in \Z^2$, $m\in \Gamma_{res}^+(n)$ if and only if $m$ and $n-m$ are orthogonal. Moreover, $m \in \Gamma_{res}^-(n)$ if and only if $n$ and $n-m$ are orthogonal.
\end{lemma}

\subsection{Explicit family of resonant frequencies and further reduction}

\myindent We shall now build a potential $(v_m(t))$ and a specific solution to $(\mathcal{RFS})$ by constructing two families $(m_k)$ and $(l_k)$, $k\geq 0$ of vectors of $\Z^2$ which satisfies good orthogonality properties. Namely, in some sense, we impose that there are no exceptional resonances.

\begin{lemma}\label{properties}
    There exists two families $(m_k)_{k\geq 0},(l_k)_{k\geq 0}$ of vectors of $\Z^2$ such that
\begin{align*}
    (P_1)\quad & m_k \neq 0,\quad l_k\neq 0 \\
    (P_2) \quad & m_k \perp l_{k'} \quad \iff \quad  k = k' \\
    (P_3) \quad & m_{k+1} = m_k + l_k \\
    (P_4) \quad & \forall k,k' \quad m_k \quad \text{is not orthogonal to}\quad m_{k'} \ \text{and is not orthogonal to} \ m_{k'} - l_{k'}\\
    (P_5) \quad &\forall k,k' \quad m_k - l_k \quad \text{is not orthogonal to} \quad l_{k'}\\
    (P_6)\quad & \forall k'\neq k+1 \quad m_{k'} - l_k \quad \text{is not orthogonal to} \quad l_k\\
    (P_7) \quad & \forall k,k' \quad m_{k'} - l_{k'} - l_k\quad \text{is not orthogonal to} \quad l_k\\
    (P_8) \quad & \forall k,k' \quad l_k + m_{k'} \quad \text{is not orthogonal to} \quad l_k \\
    (P_9) \quad & \forall k\neq k' \quad l_k + m_{k'} - l_{k'} \quad \text{is not orthogonal to} \quad l_k\\
    (P_{10}) \quad & |l_{k+1}| > |l_k| + 1 
\end{align*}

\myindent Moreover, we can find families such that there exists universal constants $C > 1 > c$ such that for all $n\geq 1$ there holds
\begin{align*}
    c(n-1)! &\leq |m_n|\leq C^n (n-1) ! \\
    cn! &\leq |l_n| \leq C^n n! 
\end{align*}

\end{lemma}

\myindent At first glance these properties may seem a lot, but it follows quite directly from geometric observations that they greatly reduces the system if we choose the potential with nonzero Fourier frequencies supported in the set $\{\pm l_k\}_{k\geq 0}$. More precisely, before proving lemma \eqref{properties}, we state and prove the following lemma 

\begin{lemma}
    Set $\Lambda := \{\pm l_k, \ k\geq 0\}$ and $\Lambda' := \{m_k,\ k\geq 0\}$. Set moreover $\Sigma := \{m_k - l_k, \ k\geq 0\}$. Assume $(a_n(t))_{n\in \Z^2}$, is a solution to $(\mathcal{RFS})$ with potential $(v_n(t))_n$ such that $(a_n(0))$ is supported in $\Lambda'\cup \Sigma$ (in the sense that $a_n(0) = 0$ whenever $n\notin\Lambda'\cup \Sigma$). If $(v_n(t))_n$ is supported in $\Lambda$ for all $t\geq 0$, then $(a_n(t))$ is supported in $\Lambda'\cup \Sigma$ for all $t\geq 0$.

    \myindent Moreover, denote $p_k(t) := a_{m_k}(t)$, $s_k(t) := a_{m_k - l_k}(t)$ and $r_k(t) := v_{l_k}(t)$ (with the convention that $p_{-1} = r_{-1} = 0$). The system $(\mathcal{RFS})$ reduces to 
    \begin{equation}\label{eqchill}
    \forall k\geq 0 \quad \begin{cases}
    \partial_t p_k =& p_{k-1} r_{k-1} - p_{k+1} r_k - s_k r_k \\
    \partial_t s_k =& p_k r_k
    \end{cases}
\end{equation}
\end{lemma}

\begin{proof}
    As $v_n(t) = 0$ whenever $n\notin \Lambda$, $(\mathcal{RFS})$ reduces to 
    \begin{equation}\label{interm}
        \partial_t a_n = \sum_{m\in \Gamma_{res}^+(n),\ n-m \in \Lambda}a_m(t)v_{n-m}(t) \ - \sum_{m\in \Gamma_{res}^-(n),\ n-m \in \Lambda}a_m(t)v_{n-m}(t)
    \end{equation}
    \myindent In order to prove the first part of the lemma, we need only show that whenever $n \notin \Lambda'\cup \Sigma$ then those $m$ that appear on the RHS of \eqref{interm} are also not in $\Lambda'\cup \Sigma$. Indeed, the system then reduces to a linear system with zero initial condition on $\Z^2\backslash \Lambda'\cup\Sigma$ so by uniqueness there stands $a_n(t) = 0$ for all $t$ whenever $n\notin\Lambda'\cup \Sigma$. 

    \myindent Take $n \notin \Lambda'\cup \Sigma$. We claim that if $m \in \Lambda'\cup \Sigma$ satisfies $n-m \in \Lambda$, then $m \notin \Gamma_{res}^+(n)\cup \Gamma_{res}^-(n)$. Indeed, assume first that $m = m_k$ for some $k$ and $n - m \in \Lambda$. Then there exists $k' \geq 0$ such that $n-m = \pm l_{k'}$.
    
    \myindent i. If $n = m_k + l_{k'}$, then $k\neq k'$ otherwise $n = m_{k+1}\in \Lambda'$, but then $m_k$ is not orthogonal to $n - m_k = l_{k'}$ thanks to $(P_2)$ thus $m_k \notin \Gamma_{res}^+(n)$. Similarly, $n-m_k = l_{k}$ is not orthogonal to $n = m_k + l_{k'}$ thanks to $(P_8)$

    \myindent ii. If $n = m_k - l_{k'}$ then $k' \neq k$ otherwise $n \in \Sigma$ and $k' \neq k-1$ otherwise $n = m_{k-1} \in \Lambda'$. Thus, $m_k \notin \Gamma_{res}^+(n)$ as $m_k$ is not orthogonal to $-l_{k'}$ by $(P_2)$, and $m_k \notin \Gamma_{res}^-(n)$ as $l_{k'} = n - m_k$ is not orthogonal to $m_k - l_{k'} = n$ thanks to $(P_6)$.

    \myindent Now, assume that $m = m_k - l_k$ for some $k\geq 0$ and $n - m = \pm l_{k'}$ for some $k' \geq 0$

    \myindent i. If $n = m_k - l_k + l_{k'}$ then $k\neq k'$ as $n\notin \Lambda'$ so $m_k - l_k$ is not orthogonal to $l_{k'}$ thanks to $(P_5)$ thus $m_k - l_k \notin \Gamma_{res}^+(n)$ and $l_{k'}$ is not orthogonal to $m_k - l_k + l_{k'}$ thanks to $(P_9)$ thus $m_k - l_k \notin \Gamma_{res}^-(n)$

    \myindent ii. Finally if $n = m_k - l_k - l_{k'}$ then as $m_k - l_k$ is not orthogonal to $-l_{k'}$ thanks to $(P_5)$ we find that $m_k - l_k \notin \Gamma_{res}^+(n)$ and as $m_k - l_k - l_{k'}$ is not orthogonal to $-l_{k'}$ thanks to $(P_7)$ we also find that $m_k - l_k \notin \Gamma_{res}^-(n)$

    \quad

    \myindent In order to prove the second part of the lemma, we follow the same steps. Take $k\geq 0$. First, let $m\in \Gamma_{res}^+(m_k)\cap (\Lambda' \cup \Sigma)$ such that $m_k - m = \pm l_{k'}$. As $m$ is orthogonal to $l_{k'}$, properties $(P_2)$ and $(P_5)$ yield that $m = m_{k'}$, thus $m_k = m_{k'}\pm l_{k'}$. As $m_k$ is orthogonal to $l_k$, $(P_5)$ yields that necessarily $m_k = m_{k'} + l_{k'} = m_{k' + 1}$ thus $k' = k- 1$ (as from $(P_3)$ and $(P_{10})$ there holds $|m_{i+1}| > |m_i|$) which yields the contribution $p_{k-1}r_{k-1}$ to the RHS of the first equation.

    \myindent Now, let $m \in \Gamma_{res}^-(m_k)\cap(\Lambda'\cup\Sigma)$ such that $m_k - m = \pm l_{k'}$. As $l_{k'}$ is orthogonal to $m_k$ $(P_2)$ yields that $k = k'$ thus $m = m_k \pm l_k$, and both are in $\Gamma_{res}^-(m_k)\cap(\Lambda'\cup\Sigma)$ This yields the contribution $-p_{k+1}r_k - s_k r_k$ to the RHS of the first equation.

    \myindent Finally, take $k\geq 0$ and $m \in \Gamma_{res}^+(m_k - l_k)\cap(\Lambda'\cup \Sigma)$ such that $m_k - l_k - m = \pm l_{k'}$ : as $m$ is orthogonal to $l_{k'}$ we find again that $m = m_{k'}$ and thus that $m_k = l_k + m_{k'} \pm l_{k'}$. If the sign is a minus, properties $(P_9)$ yields that $k' = k$ and thus $m = m_k$ which gives the contribution $p_k r_k$ to the RHS of the second equation (as we remind that $v_{-n} = v_n$ for all $n$). If the sign is a plus, we find that $m_k - l_k = m_{k' + 1}$ is orthogonal to $l_{k' + 1}$ which contradicts properties $(P_5)$.

    \myindent We see moreover that there is not a $m\in\Gamma_{res}^-(m_k - l_k)\cap(\Lambda'\cup \Sigma)$ such that $m_k - l_k - m \in \Lambda$. Indeed by definition of $\Gamma_{res}^-$, this would mean that there is a $k'$ such that $m_k - l_k$ is orthogonal to $l_{k'}$ thus contradicting property $(P_5)$.
\end{proof}

\myindent We now turn to the proof of lemma \eqref{properties}. Choose $m_0 \in \Z^2 \backslash \{0\}$ arbitrarily, for example $m_0 = (1,0)$. As $m_{k+1} = m_k + l_k$ we need only construct the $l_k$ for $k\geq 0$. We will do so by induction. Assume the sequence $(m_k)$ is constructed up to $k = n$ satisfying the properties (which means that $l_0,...,l_{n-1}$ have been constructed). We need to exhibit $l_n\in \Z^2$ (and thus $m_{n+1} = m_n + l_n$) such that the properties still hold up to $k = n + 1$. Define $m$ the vector obtained from $m_n$ by applying a rotation with angle $\pi/2$ (which is orthogonal to $m_n$ and which has the same Euclidean norm). We will show that there is $a \in \N$ such that $n +1 \leq a \leq C(n+1)$ with $C$ a universal constant such that setting $l_n := am$ will do.

\myindent $(P_1)$ always holds.

\myindent In order for $(P_2)$ to hold, observe first that, by construction, $l_n = am$ is orthogonal to $m_n$. Moroever, we need on the one hand that $m_k$ is not orthogonal to $l_n$ for $k\leq n-1$. However, since $m_k \perp l_k$ up to $k = n$ by induction, and since we are in dimension two, this amounts to asking that $l_k$ is not orthogonal to $m_n$ for $k\leq n-1$, which is true by induction by $(P_2)$. On the other hand, we need that $m_{n+1}$ is not orthogonal to $l_k$ up to $k = n$, that is, since $l_k \perp m_k$, and since we are in dimension two, that we need only prove that $m_{n+1} = m_n + am$ is not parallel to $m_k$ for $k\leq n$. It is always true for $k = n$ as $a > 0$ ; and for each $k\leq n - 1$ there is at most one value of $a$ for which $m_{n+1}$ could be parallel to $m_k$ (as $m$ is not parallel to $m_k$ otherwise $m_k$ would be orthogonal to $m_n$ thus contradicting $(P_4)$). This excludes at mot $n$ possible values for $a$.

\myindent In order for $(P_4)$ to hold, we need that $m_{n+1}\cdot m_k \neq 0$ for $k\leq n$. It is always true for $k = n$, and for $k < n$ it means that $m_n\cdot m_k + a m \cdot m_k \neq 0$. Now, $m\cdot m_k \neq 0$ otherwise this would contradict $(P_2)$. Thus, at most $n$ possible values of $a$ are to be excluded. We also need that $m_{n+1}\cdot (m_k - l_k) \neq 0$ for $k \leq n$, which is always true for $k = n$ if we set $a \geq 2$, and it follows from the construction of $(P_5)$ that $m\cdot(m_k - l_k)\neq 0$ as $m_n$ is not parallel to $m_k - l_k$, hence this excludes at most $n$ values of $a$. We finally need that $m_k \cdot (m_n - am) \neq 0$ for $k\leq n-1$ which excludes at most $n$ values of $a$ as $m\cdot m_k \neq 0$.

\myindent In order for $(P_5)$ to hold, we need on the one hand that $m_k - l_k$ is not parallel to $m_{n+1}$ for $k\leq n$, which excludes at most $n$ values of $a$ as this is always true for $k = n$ and as $m_k - l_k$ is not parallel to $m$ for $k < n$ thanks to $(P_4)$. On the other hand, we need that $m_n - l_n = m_n - am$ is not parallel to $m_k$ for $k\leq n$ which again excludes at most $n$ values for $a$ as $m$ is not parallel to $m_k$ for $k < n$ thanks to $(P_2)$.

\myindent In order for $(P_6)$ to hold, we need on the one hand that for $k\leq n-1$, $(m_{n+1} - l_k)\cdot l_k \neq 0$ which is equivalent to $am\cdot l_k \neq cste$. As we know that $m\cdot l_k \neq 0$ (otherwise $m_k$ is orthogonal to $m_n$) this excludes at most $n$ values for $a$. One the other hand, we need that for $k\leq n - 1$ $m_k - am$ is not orthogonal to $am$ which is ensured by the fact that $|m| > |m_k|$.

\myindent In order for $(P_7)$ to hold, we need on the one hand that for $k\leq n-1$ $m_n - am - l_k$ is not orthogonal to $l_k$ which excludes at most $n$ values for $a$ as $m\cdot l_k \neq 0$. On the other hand we need that for $k\leq n-1$ $m_k - l_k - am$ is not orthogonal to $am$ and once again this excludes at most $n$ values for $a$.

\myindent In order for $(P_8)$ to hold, we need on the one hand that for $k\leq n$ $am + m_k$ is not orthogonal to $am$, thus excluding at most $n$ values for $a$, and on the other hand that for $k\leq n-1$ $l_k + m_n + am$ is not orthogonal to $l_k$ which excludes at most $n$ values for $a$ as $m\cdot l_k \neq 0$.

\myindent In order for $(P_9)$ to hold finally, we need on the one hand that for $k\leq n-1$ $am + m_k - l_k$ is not orthogonal to $am$, thus excluding at most $n$ values for $a$, and on the other hand that $l_k + m_n - am$ is not orthogonal to $l_k$, excluding once again at most $n$ values for $a$.

\myindent We thus finally see that any $a \geq 1$ except maybe at most $C(n+1)$ values can be chosen, where $C\geq 2$. Up to taking $C$ a little larger we may thus find $n + 1\leq a \leq C (n + 1)$ such that setting $l_n = am $ enables the induction hypothesis to be satisfied.

\myindent By this procedure we are able to construct sequences for which the desired properties hold. Moreover there holds $n|m_n| \leq |l_n| \leq Cn |m_n|$ and thus $n|m_n| \leq |m_{n+1}|\leq C'n |m_n|$ for $C' = \sqrt{C + 1}$ thus proving the last part of the lemma.

\section{Solution to the resonant system}

\myindent Thanks to the previous section, we are now able to exhibit explicit $(r_k(t))$ and an explicit solution $(p_k(t)), (s_k(t))$ to \eqref{eqchill} for which we control precisely the energy transfer between Fourier frequencies.
We turn to the explicit study of the mechanism that will allow energy transfer between frequencies. We start at $t = 0$ with well-chosen values for $p_0, p_1, s_0$ and set the other $p_k$ and $s_k$ to be zero. The idea is then to locally fully transfer the energy from $(p_k, s_k, p_{k+1})$ to $(p_{k+1}, s_{k+1}, p_{k+2})$ in finite time, thus ensuring that for all given $n$, after a time $T_n$, we have that $p_k = s_k = 0$ for all $k\leq n$. Now, as $(\mathcal{RFS})$ conserves the $l^2$ norm, this ensures that the Sobolev $H^s$ norm is greater that $|m_n|^s$ for $t\geq T_n$.

\subsection{General form of the solution to the linear system}

\myindent Explicitly, find an interval $I = [t_0, t_1]$ and a smooth function $\phi$ on $I$. Find $k\geq 1$. We look at the system

\begin{equation}
    \begin{cases}
    \partial_t p_{k+1}& = \phi(t) p_k \\
    \partial_t p_k &= -\phi(t) p_{k+1} - \phi(t) s_k\\
    \partial_t s_k &= \phi(t) p_k
    \end{cases}
\end{equation}

which corresponds to \eqref{eqchill} when we only light up $r_k(t) = \phi(t)$, that is we set $r_{k'}(t) = 0$ for $k'\neq k$ on $I$. The system can then be written in the form of a simple linear system

\begin{equation}
    \partial_t \begin{pmatrix}
    p_{k+1} \\
    p_k \\
    s_k \end{pmatrix} = \phi(t)A \begin{pmatrix}
    p_{k+1} \\
    p_k \\
    s_k \end{pmatrix}
\end{equation}

where we set

\begin{equation}
    A = \begin{pmatrix}
    0 & 1 & 0 \\
    -1 & 0 & - 1\\
    0 & 1 & 0 \end{pmatrix}
\end{equation}

\myindent Now, the solution with initial condition $\begin{pmatrix}
    p_{k+1}(t_0) \\
    p_k (t_0)\\
    s_k(t_0) \end{pmatrix}$ is given by
    
    \begin{equation}
        \begin{pmatrix}
    p_{k+1}(t) \\
    p_k(t) \\
    s_k(t) \end{pmatrix} = \exp\left(\left(\int_{t_0}^t \phi(s)ds\right) A \right)\begin{pmatrix}
    p_{k+1}(t_0) \\
    p_k(t_0) \\
    s_k(t_0) \end{pmatrix}
    \end{equation}
    
\myindent Now, one can compute

\begin{equation}
    \exp(T A) = \begin{pmatrix} 
    \frac{1}{2}\left(\cos(T\sqrt{2}) + 1\right) & \frac{1}{\sqrt{2}}\sin(T\sqrt{2}) & \frac{1}{2}\left(\cos(T\sqrt{2}) - 1\right)\\
    -\frac{1}{\sqrt{2}}\sin(T\sqrt{2}) & \cos(T\sqrt{2}) & -\frac{1}{\sqrt{2}}\sin(T\sqrt{2})\\
    \frac{1}{2}\left(\cos(T\sqrt{2}) - 1\right) & \frac{1}{\sqrt{2}}\sin(T\sqrt{2}) & \frac{1}{2}\left(\cos(T\sqrt{2}) + 1\right) \end{pmatrix}
\end{equation}

\myindent This explicit matrix allows us to build three moves  in order to transfer a specific configuration from $(p_k, s_k, p_{k+1})$ to $(p_{k+1}, s_{k+1}, p_{k+2})$ in finite time.

\subsubsection{First move}

\myindent Start with

\begin{equation}
    \begin{pmatrix}
    p_{k+1}(t_0) \\
    p_k(t_0) \\
    s_k(t_0) \end{pmatrix} = \begin{pmatrix} \frac{1}{2} \\ -\frac{1}{\sqrt{2}} \\ \frac{1}{2} \end{pmatrix}
\end{equation}

\myindent We set $\phi$ a non-negative $\mathcal{C}^{\infty}$ function with support in $[t_0,t_1]$ such that moreover $\int \phi = \frac{7\pi}{4\sqrt{2}}$. There holds

\begin{equation}
    \begin{pmatrix}
    p_{k+1}(t_1) \\
    p_k(t_1) \\
    s_k(t_1) \end{pmatrix} = \begin{pmatrix} \frac{1}{\sqrt{2}} \\ 0 \\ \frac{1}{\sqrt{2}}\end{pmatrix}
\end{equation}

\subsubsection{Second move}

\myindent If we now set

\begin{equation}
    \begin{pmatrix}
    p_{k+1}(t_0) \\
    p_k(t_0) \\
    s_k(t_0) \end{pmatrix} = \begin{pmatrix} 0 \\ 1 \\ 0 \end{pmatrix}
\end{equation}

\myindent With the integral of $\phi$ being $\frac{\pi}{2\sqrt{2}}$ there holds

\begin{equation}
    \begin{pmatrix}
    p_{k+1}(t_1) \\
    p_k(t_1) \\
    s_k(t_1) \end{pmatrix} = \begin{pmatrix} \frac{1}{\sqrt{2}} \\ 0 \\ \frac{1}{\sqrt{2}} \end{pmatrix}
\end{equation}

\subsubsection{Third move}

\myindent If finally we set 

\begin{equation}
    \begin{pmatrix}
    p_{k+1}(t_0) \\
    p_k(t_0) \\
    s_k(t_0) \end{pmatrix} = \begin{pmatrix} 0 \\ 0 \\ 1 \end{pmatrix}
\end{equation}

and set the integral of $\phi$ to be $\frac{\pi}{\sqrt{2}}$ there holds

\begin{equation}
    \begin{pmatrix}
    p_{k+1}(t_1) \\
    p_k(t_1) \\
    s_k(t_1) \end{pmatrix} = \begin{pmatrix} -1 \\ 0 \\ 0\end{pmatrix}
\end{equation}

\subsection{Idea of the construction of the potential and the resonant solution}

\myindent These easy observations yield the construction both of the potential $(r_k(t))$ and of the solution $(p_k(t)),(s_k(t))$. We may represent the solution $(p_k(t)),(s_k(t))$ as points in the following semi-infinite chain

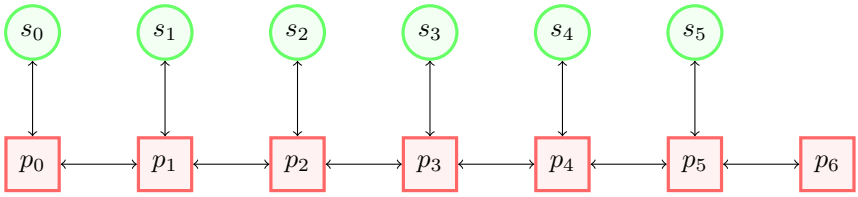
\quad

\begin{tikzpicture}[
roundnode/.style={circle, draw=green!60, fill=green!5, very thick, minimum size=7mm},
squarednode/.style={rectangle, draw=red!60, fill=red!5, very thick, minimum size=7mm},
]
\node[squarednode]      (p0)                              {$p_0$};
\node[roundnode]        (s0)       [above=of p0] {$s_0$};
\node[squarednode]      (p1)       [right=of p0] {$p_1$};
\node[roundnode]        (s1)       [above=of p1] {$s_1$};
\node[squarednode]      (p2)       [right=of p1]                       {$p_2$};
\node[roundnode]        (s2)       [above=of p2] {$s_2$};
\node[squarednode]      (p3)       [right=of p2] {$p_3$};
\node[roundnode]        (s3)       [above=of p3] {$s_3$};
\node[squarednode]      (p4)       [right=of p3]                       {$p_4$};
\node[roundnode]        (s4)       [above=of p4] {$s_4$};
\node[squarednode]      (p5)       [right=of p4] {$p_5$};
\node[roundnode]        (s5)       [above=of p5] {$s_5$};
\node[squarednode]      (p6)       [right=of p5] {$p_6$};

\draw[<->] (s0.south) -- (p0.north);
\draw[<->] (p0.east) -- (p1.west);
\draw[<->] (p1.north) -- (s1.south);
\draw[<->] (p1.east) -- (p2.west);
\draw[<->] (p2.north) -- (s2.south);\draw[<->] (p2.east) -- (p3.west);
\draw[<->] (p3.north) -- (s3.south);\draw[<->] (p3.east) -- (p4.west);
\draw[<->] (p4.north) -- (s4.south);\draw[<->] (p4.east) -- (p5.west);
\draw[<->] (p5.north) -- (s5.south);\draw[<->] (p5.east) -- (p6.west);
\end{tikzpicture}

\quad

where the arrows represent the possible interactions between the Fourier frequencies induced by the potential $(r_k(t))$

\myindent Assume that, at $t=0$, we start with the configuration 

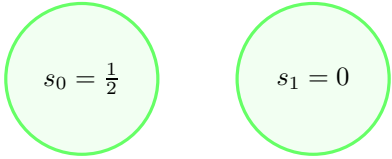
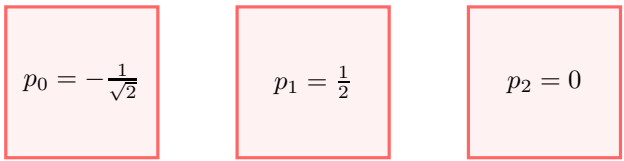
\quad

\begin{tikzpicture}[
roundnode/.style={circle, draw=green!60, fill=green!5, very thick, minimum size=20mm},
squarednode/.style={rectangle, draw=red!60, fill=red!5, very thick, minimum size=20mm},
]

\node[squarednode]      (p0)                              {$p_0 = -\frac{1}{\sqrt{2}}$};
\node[roundnode]        (s0)       [above=of p0] {$s_0 = \frac{1}{2}$};
\node[squarednode]      (p1)       [right=of p0] {$p_1 = \frac{1}{2}$};
\node[roundnode]        (s1)       [above=of p1] {$s_1 = 0$};
\node[squarednode]      (p2)       [right=of p1]                       {$p_2 = 0$};

\end{tikzpicture}

\quad

\myindent Then using first move if we light up only $r_0$ during an appropriate time we may fully transfer the mass from $p_0$ to $s_0$ and $p_1$ equally.

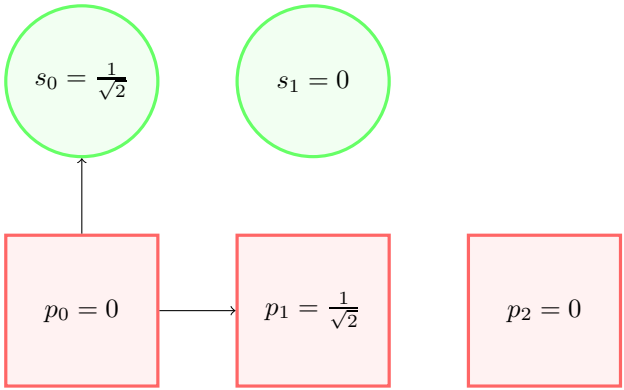
\quad

\begin{tikzpicture}[
roundnode/.style={circle, draw=green!60, fill=green!5, very thick, minimum size=20mm},
squarednode/.style={rectangle, draw=red!60, fill=red!5, very thick, minimum size=20mm},
]

\node[squarednode]      (p0)                              {$p_0 = 0$};
\node[roundnode]        (s0)       [above=of p0] {$s_0 = \frac{1}{\sqrt{2}}$};
\node[squarednode]      (p1)       [right=of p0] {$p_1 = \frac{1}{\sqrt{2}}$};
\node[roundnode]        (s1)       [above=of p1] {$s_1 = 0$};
\node[squarednode]      (p2)       [right=of p1]                       {$p_2 = 0$};

\draw[<-] (s0.south) -- (p0.north);
\draw[->] (p0.east) -- (p1.west);

\end{tikzpicture}

\quad

\myindent Now, we clear $p_1$ using the second move, that is lighting up only $r_1(t)$ we can fully transfer the mass from $p_1$ to $s_1$ and $p_2$ equally.

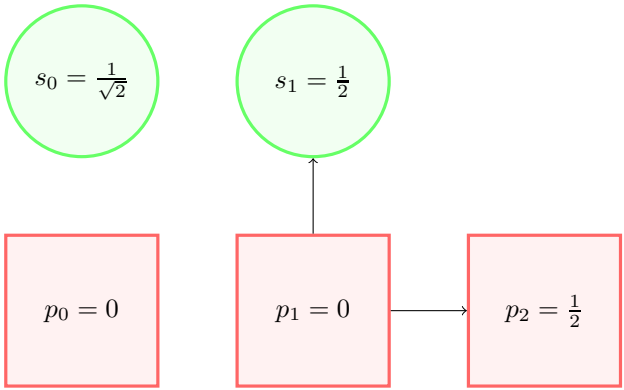
\quad

\begin{tikzpicture}[
roundnode/.style={circle, draw=green!60, fill=green!5, very thick, minimum size=20mm},
squarednode/.style={rectangle, draw=red!60, fill=red!5, very thick, minimum size=20mm},
]

\node[squarednode]      (p0)                              {$p_0 = 0$};
\node[roundnode]        (s0)       [above=of p0] {$s_0 = \frac{1}{\sqrt{2}}$};
\node[squarednode]      (p1)       [right=of p0] {$p_1 = 0$};
\node[roundnode]        (s1)       [above=of p1] {$s_1 = \frac{1}{2}$};
\node[squarednode]      (p2)       [right=of p1]                       {$p_2 = \frac{1}{2}$};

\draw[<-] (s1.south) -- (p1.north);
\draw[->] (p1.east) -- (p2.west);

\end{tikzpicture}

\quad

\myindent Finally, we use third move to transfer fully the remaining mass from $s_0$ to $p_1$ lighting only $r_0$ again

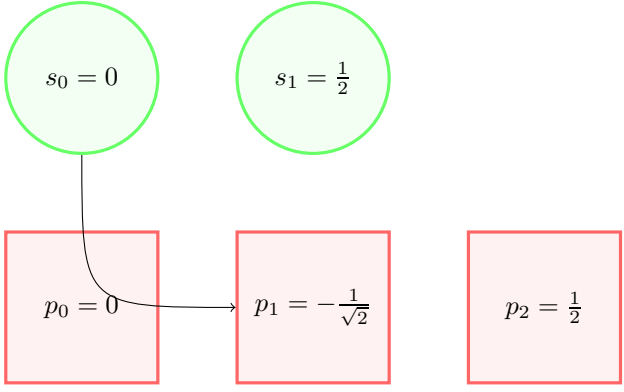
\quad

\begin{tikzpicture}[
roundnode/.style={circle, draw=green!60, fill=green!5, very thick, minimum size=20mm},
squarednode/.style={rectangle, draw=red!60, fill=red!5, very thick, minimum size=20mm},
]

\node[squarednode]      (p0)                              {$p_0 = 0$};
\node[roundnode]        (s0)       [above=of p0] {$s_0 = 0$};
\node[squarednode]      (p1)       [right=of p0] {$p_1 = -\frac{1}{\sqrt{2}}$};
\node[roundnode]        (s1)       [above=of p1] {$s_1 = \frac{1}{2}$};
\node[squarednode]      (p2)       [right=of p1]                       {$p_2 = \frac{1}{2}$};

\draw[->] (s0.south) .. controls (p0) .. (p1.west);

\end{tikzpicture}

\quad

\myindent Thus, we find exactly the same situation we started with with indexes incremented by one. This enables us to start a recursive scheme so that as time goes by we repeat these three moves to transfer the mass to higher frequencies. The idea to ensure that the potential $V$ decreases in Sobolev norms as $t\to\infty$ is that, up to lighting $r_k$ for a longer time, we may at each step choose it arbitrarily small.

\subsection{Explicit computation of the potential and of the resonant solution}

\myindent We now make the previous argument rigorous. We first find a smooth function $\phi$ on $\R$, nonnegative and nondecreasing, such that $\phi = 0$ on $(-\infty, 0]$, $\phi = 1$ on $[1, + \infty)$, and we set $\alpha = \int_0^1 \phi$. Take $(\beta_k)_{k\geq0}$ a sequence of positive real numbers such that $\beta_k \ll 1$. The $(\beta_k)$ will control the amplitude to which we light up $r_k$, and we will fix them later in order to control the decay of the potential $V$ in Sobolev norms.

\subsubsection{Initialling the induction}

\myindent We choose at $t = 0$

\begin{equation}
    \begin{pmatrix}
    p_1(0) \\
    p_0(0)\\
    s_0(0) 
    \end{pmatrix}= \begin{pmatrix} \frac{1}{2} \\ -\frac{1}{\sqrt{2}} \\ \frac{1}{2} \end{pmatrix}
\end{equation}

(and the other $p_k,s_k$ are set to zero). We now set

\begin{equation}
    \begin{cases}
    r_0(t) &= \frac{7\pi}{4\sqrt{2}\alpha} \beta_0 \phi(t) \quad 0\leq t \leq 1 \\
    r_0(t) &= \frac{7\pi}{4\sqrt{2}\alpha} \beta_0 \quad 1 \leq t \leq 1 + t_0\\
    r_0(t) &= \frac{7\pi}{4\sqrt{2}\alpha} \beta_0 \phi(t_0 + 2 - t) \quad 1 + t_0 \leq t \leq t_0 + 2
    \end{cases}
\end{equation}

where we set $t_0$ such that $\int_0^{t_0 +2} r_0 = \frac{7\pi}{4\sqrt{2}}$, which means $t_0 = \alpha (\beta_0^{-1} - 2)$. We set $r_k(t) = 0$ on $[0,t_0 + 2 ]$ for all $k\geq 1$. 

\myindent Now, at $t = t_0 + 2$, we find 

\begin{equation}
    \begin{pmatrix}
    p_1(t_0 + 2) \\
    p_0(t_0 + 2)\\
    s_0(t_0 + 2) 
    \end{pmatrix}= \begin{pmatrix} \frac{1}{\sqrt{2}} \\ 0 \\ \frac{1}{\sqrt{2}} \end{pmatrix}
\end{equation}

\myindent Set now

\begin{equation}
    \begin{cases}
    r_1(t) &= \frac{\pi}{2\sqrt{2}\alpha} \beta_1 \phi(t - (t_0 + 2)) \quad t_0 + 2 \leq t \leq t_0 + 3\\
    r_1(t) &= \frac{\pi}{2\sqrt{2}\alpha} \beta_1 \quad t_0 + 3 \leq t \leq t_0 + 3 + t_1 \\
    r_1(t) &= \frac{\pi}{2\sqrt{2}\alpha} \beta_1 \phi(4 + t_0  + t_1 - t) \quad 3 + t_0 +  t_1 \leq t \leq 4 + t_0 + t_1
    \end{cases}
\end{equation}

with $t_1$ such that the integral of $r_1$ on $[2 + t_0, 4 + t_0 + t_1]$ is equal to $\frac{\pi}{2\sqrt{2}}$, which means $t_1 = \alpha(\beta_1^{-1} - 2)$. Set $r_k(t) = 0$ on $[2+ t_0, 4 + t_0+ t_1]$ for all $k\neq 1$. Now, at $t = 4 + t_0 + t_1$, there holds

\begin{equation}
    \begin{pmatrix}
    p_2(4 + t_0 + t_1) \\
    p_1(4 + t_0 + t_1) \\
    s_1(4 + t_0 + t_1) \\
    p_0(4 + t_0 + t_1) \\
    s_0(4 + t_0 + t_1)
    \end{pmatrix} 
    = \begin{pmatrix}
    \frac{1}{2} \\
    0 \\
    \frac{1}{2}\\
    0 \\
    \frac{1}{\sqrt{2}}
    \end{pmatrix}
\end{equation}

and the other $p_k,s_k$ are equal to zero. To finish the cycle we need to transfer all the mass from $s_0$ to $p_1$, and we will end up with $(p_2,p_1,s_1) = (\frac{1}{2}, - \frac{1}{\sqrt{2}}, \frac{1}{2})$ wich was exactly the initial state on $(p_1,p_0,s_0)$. This enables to start a recursive process. More precisely, set

\begin{equation}
    \begin{cases}
    r_0(t) &= \frac{\pi}{\sqrt{2}\alpha} \beta_0 \phi(t - (4 + t_0 + t_1))\quad 4 + t_0 + t_1 \leq t \leq 5 + t_0 + t_1 \\
    r_0(t) &= \frac{\pi}{\sqrt{2}\alpha} \beta_0 \quad 5 + t_0 + t_1 \leq t \leq 5 + t_0 + t_1 + t_0 \\
    r_0(t) &= \frac{\pi}{\sqrt{2}\alpha} \beta_0 \phi(6 + t_0 + t_1 + t_0 - t) \quad 5 + t_0 + t_1 + t_0 \leq t \leq 6 + t_0 + t_1 + t_0
    \end{cases}
\end{equation}

with once again $t_0 = \alpha(\beta_0^{-1} - 2)$, and $r_k(t) = 0$ on $[4 + t_0 + t_1, 6 + t_0 + t_1 + t_0]$ for $k\neq 1$. There holds at $t = 6 + 2t_0 + t_1$

\begin{equation}
    \begin{pmatrix}
    p_2(t) \\
    p_1(t) \\
    s_1(t) \\
    p_0(t) \\
    s_0(t) 
    \end{pmatrix}
    = \begin{pmatrix}
    \frac{1}{2} \\
    -\frac{1}{\sqrt{2}} \\
    \frac{1}{2}\\
    0 \\
    0
    \end{pmatrix}
\end{equation}

as was expected. 

\subsubsection{Recursive scheme}

\myindent Now, set $t_n := \alpha(\beta_n^{-1} - 2)$, and suppose that there holds for $T_n = 6n + 2t_0 + 3t_1 + 3t_2 + ... + 3t_{n-1} + t_n$
\begin{equation}
    \begin{pmatrix}
    p_{n+1} \\
    p_n\\
    s_n
    \end{pmatrix}
    = \begin{pmatrix}
    \frac{1}{2}\\
    -\frac{1}{\sqrt{2}}\\
    \frac{1}{2}
    \end{pmatrix}
\end{equation}

with the other $p_k,s_k$ being equal to zero. We set now

\begin{equation}
    \begin{cases}
    r_n(t) &= \frac{7\pi}{8\sqrt{2}\alpha} \beta_n \phi(t) \quad T_n \leq t \leq 1 + T_n\\
    r_n(t) &= \frac{7\pi}{8\sqrt{2}\alpha} \beta_n \quad 1 + T_n \leq t \leq 1 + T_n + t_n\\
    r_n(t) &= \frac{7\pi}{8\sqrt{2}\alpha} \beta_n \phi(2 + T_n + t_n - t) \quad 1 + T_n + t_n \leq t \leq T_n + 2 + t_n
    \end{cases}
\end{equation}

all the other $r_k$ being set to zero on $[T_n, T_n + 2 + t_n]$. Now there holds at $t = T_n + 2 + t_n$ 

\begin{equation}
    \begin{pmatrix}
    p_{n+1} \\
    p_n\\
    s_n
    \end{pmatrix}
    = \begin{pmatrix}
    \frac{1}{\sqrt{2}}\\
    0\\
    \frac{1}{\sqrt{2}}
    \end{pmatrix}
\end{equation}

\myindent Now set

\begin{equation}
    \begin{cases}
    r_{n+1}(t) &= \frac{\pi}{2\sqrt{2}\alpha} \beta_{n+1} \phi(t - (T_n + 2 + t_n)) \quad T_n + 2 + t_n \leq t \leq T_n + 3 + t_n\\
    r_{n+1}(t) &= \frac{\pi}{2\sqrt{2}\alpha} \beta_{n+1} \quad T_n + 3 + t_n \leq t \leq T_n + 3 + t_n + t_{n+1} \\
    r_{n+1}(t) &= \frac{\pi}{2\sqrt{2}\alpha} \beta_{n+1} \phi(T_n + 4 + t_n + t_{n+1}- t) \quad T_n + 3 + t_n  + t_{n+1} \leq t \leq T_n + 4 + t_n + t_{n+1}
    \end{cases}
\end{equation}

the other $r_k$ being set to zero on $[T_n + 2 + t_n, T_n + 4 + t_n + t_{n+1}]$. There holds at $t = T_n + 4 + t_n + t_{n+1}$ 

\begin{equation}
    \begin{pmatrix}
    p_{n+2}(T_n + 4 + t_n + t_{n+1}) \\
    p_{n+1}(T_n + 4 + t_n + t_{n+1}) \\
    s_{n+1}(T_n + 4 + t_n + t_{n+1}) \\
    p_n(T_n + 4 + t_n + t_{n+1}) \\
    s_n(T_n + 4 + t_n + t_{n+1})
    \end{pmatrix} 
    = \begin{pmatrix}
    \frac{1}{2} \\
    0 \\
    \frac{1}{2} \\
    0 \\
    \frac{1}{\sqrt{2}}
    \end{pmatrix}
\end{equation}

\myindent Set finally

\begin{equation}
    \begin{cases}
    r_n(t) &= \frac{\pi}{\sqrt{2}\alpha} \beta_n \phi(t - (T_n + 4 + t_n + t_{n+1}))\quad T_n + 4 + t_n + t_{n+1} \leq t \leq T_n + 5 + t_n + t_{n+1} \\
    r_n(t) &= \frac{\pi}{\sqrt{2}\alpha} \beta_n \quad T_n + 5 + t_n + t_{n+1}\leq t \leq T_n + 5 + t_n + t_{n+1} + t_n \\
    r_n(t) &= \frac{\pi}{\sqrt{2}\alpha} \beta_n \phi(T_n + 6 + t_n + t_{n+1} + t_n - t) \quad T_n + 5 + t_n + t_{n+1} + t_n \leq t \leq T_n + 6 + t_n + t_{n+1} + t_n
    \end{cases}
\end{equation}

\myindent There now holds at $T_{n+1} = T_n + 6 + t_n + t_{n+1} + t_n$

\begin{equation}
    \begin{pmatrix}
    p_{n+2}(t) \\
    p_{n+1}(t) \\
    s_{n+1}(t) \\
    p_n(t) \\
    s_n(t) 
    \end{pmatrix}
    = \begin{pmatrix}
    \frac{1}{2} \\
    -\frac{1}{\sqrt{2}} \\
    \frac{1}{2}\\
    0 \\
    0
    \end{pmatrix}
\end{equation}

\myindent We may now induce this construction for all $n\geq 1$, which yields a solution $(p_k(t),s_k(t))$ to \eqref{eqchill}, thus leading to a solution $(a_n(t))$ of $(\mathcal{RFS})$ which we control very explicitly.

\begin{remark}\label{behavioura}
    Provided the $\beta_k$'s are small enough, the explicit construction yields firstly that $|a_n(t)| \leq 1$ for all $n,t$, and secondly the following behaviour for $(a_n(t))$ : for each n, observe that $a_n(t) = 0$ outside of a finite interval. Moreover, this interval can be divided into a bounded number of subintervals so that either those subintervals are of length $2$ (corresponding to the time we take in order to light up a $r_k$ or light it out), either $a_n(t)$ is a finite linear combination of oscillating factors $e^{ift}$, where the frequency $f$ is of the order of $\beta_k$ for some $k$, hence arbitrarily small. 
\end{remark}

\subsection{Explicit choice for $\beta_k$ in order for $V$ to decay}

\myindent In order to prove theorem $\eqref{result}$, we need to ensure that $V$ and all its derivative decay with respect to all Sobolev norms as $t\to\infty$. Now, from the construction we see that for all $t\geq 0$ there is a unique $k(t)$ such that $v_{n} = 0$ for all $n\neq \pm l_{k(t)}$. Now, find a $m\in \N$ and a $s \geq 0$, there holds

\begin{equation}
    \|\partial^m_t V(t,\cdot)\|_{H^s} \simeq \beta_{k(t)}\left|l_{k(t)}\right|^{s + 2m}
\end{equation}

\myindent As $k(t) \to + \infty$ when $t\to + \infty$, and thus as $|l_{k(t)}| \to + \infty$, we need to ensure that $\beta_k$ decays faster with respect to $k$ that any power of $l_k$. A natural choice is

\begin{equation}
    \beta_k := |l_k|^{-|l_k|}
\end{equation}

and we will see that this choice indeed enables us to close the estimates.

\section{Approximation}

\subsection{Resonant solution + perturbation decomposition}

\myindent In order to construct a solution to the full system $(\mathcal{FS})$ we try and approximate it by the solution $(a_n(t))$ built in the previous section. In that spirit, we set the solution $(b_n(t))$ with the a priori form $b_n(t) = a_n(t) + c_n(t)$  where $a_n(t)$ is the solution to $(\mathcal{RFS})$ built above and $c_n$ is a perturbation. We may thus write

\begin{equation}
    \partial_t (a_n + c_n) = \sum_{m \in \Z^2} (a_m + c_m)(t) v_{n-m}(t) \left(e^{-i\omega_{m,n}^+t} - e^{-i\omega_{m,n}^-t}\right)
\end{equation}

and we already know that

\begin{equation}
    \partial_t a_n = \sum_{m\in \Gamma_{res}^+(n)}a_m(t)v_{n-m}(t) \ - \sum_{m\in \Gamma_{res}^-(n)}a_m(t)v_{n-m}(t)
\end{equation}

\myindent Thus we need $(c_n)$ to solve 

\begin{equation}\label{eqpert}
    \partial_t c_n = \sum_{m \in \Z^2} c_m(t) v_{n-m}(t) \left(e^{-i\omega_{m,n}^+t} - e^{-i\omega_{m,n}^-t}\right)
    + \sum_{m\notin \Gamma_{res}^+(n)}a_m(t)v_{n-m}(t)e^{-i\omega_{m,n}^+t} \ - \sum_{m\notin \Gamma_{res}^-(n)}a_m(t)v_{n-m}(t)e^{-i\omega_{m,n}^-t}
\end{equation}

\myindent Our goal is now to build a solution $(c_n)$ to \eqref{eqpert} which decays as $t\to \infty$. We will use a Cauchy sequence method: the equation \eqref{eqpert} is globally well-posed in $l^1(\Z)$ so we may set for a given integer $N>0$ $(c_n^N)$ the solution on $\R_+$ with initial condition $c^N(T_N) = 0$. There holds

\begin{align*}
    c_n^N(t) &= -\sum_{m\in \Z^2} \int_t^{T_N} c_m^N(s) v_{n-m}(s) \left(e^{-i\omega_{m,n}^+s} - e^{-i\omega_{m,n}^-s}\right) ds \\
    &- \sum_{m\notin \Gamma_{res}^+(n)} \int_t^{T_N} a_m(s)v_{n-m}(s)e^{-i\omega_{m,n}^+s}ds \ + \sum_{m\notin \Gamma_{res}^-(n)}\int_t^{T_N}a_m(s)v_{n-m}(s)e^{-i\omega_{m,n}^-s} ds
\end{align*}

from which we infer, for $t\leq T_N$

\begin{align*}
    \|(c_n^N(t))\|_{l^1} &\leq 2\int_t^{T_N} \|(c_n^N(s))\|_{l^1} \|(v_n(s))\|_{l^1} ds \\
    &+ \sum_{n} \sum_{m\notin \Gamma_{res}^+(n)} \left|\int_t^{T_N} a_m(s)v_{n-m}(s)e^{i\omega_{m,n}^+s}ds\right| \ + \sum_{n} \sum_{m\notin \Gamma_{res}^-(n)} \left|\int_t^{T_N} a_m(s)v_{n-m}(s)e^{i\omega_{m,n}^-s}ds\right|
\end{align*}

which we rewrite as the inequality for $t\leq T_N$

\begin{equation}
    \|(c_n^N(t))\|_{l^1} \leq \alpha(t) + \int_t^{T_N} \|(c_n^N(s))\|_{l^1} \beta(s) ds
\end{equation}

\myindent By Gronwall's lemma

\begin{equation}
    \forall t \leq T_N \quad \|(c_n^N(t))\|_{l^1} \leq \alpha(t) + \int_t^{T_N} \alpha(s) \beta(s) \exp\left(\int_t^s \beta(\sigma) d\sigma\right) ds
\end{equation}

\subsection{Estimates on $\alpha(t)$}

\myindent First, let us study $\alpha(t)$. The set of pairs $(m,n-m)$, $n \in Z^2$ and $m\notin \Gamma_{res}^+(n)$ (resp $m\notin \Gamma_{res}^-(n)$) is equal to the set of pairs $(n_1,n_2)\in \Z^2$ such that $n_1$ et $n_2$ aren't orthogonal (resp $n_2$ and $n_1 + n_2$ aren't orthogonal). Moreover there stands $v_n(s) = 0$ for all $n \neq \pm l_k$ for a given $k\geq 0$, and we recall that $v_{-n} = v_n$. Finally we know that $a_n(s) v_{l_k}(s) = 0$ as soon as $n \notin \{m_k,m_k - l_k, m_{k+1}, m_{k+1} - l_{k+1}, m_{k+2}\} =: E_k$. We may then write 

\begin{equation}
    \alpha(t) = \sum_{k\geq 0} \sum_{n \in E_k} I(k,n,t)
\end{equation}

where $I(k,n)$ is a sum of at most four quantities of the form 

\begin{equation}
    J(k,n,\omega,t) := \left|\int_t^{T_n} a_n(s)r_k(s) e^{i\omega s}ds\right|
\end{equation}

and $\omega$ is a frequency belonging to $\Z\backslash \{0\}$, thus ensuring $|\omega| \geq 1$. (This is here that we use the non resonance of the interactions).

\myindent We may now write 

\begin{align*}
    \int_t^{T_n} a_n(s)r_k(s) e^{i\omega s}ds &= \left[\left(\int_s^ta_n(\sigma)e^{i\omega \sigma}d\sigma \right) r_k(s)\right]_t^{T_N} - \int_t^{T_N}\left(\int_s^ta_n(\sigma)e^{i\omega \sigma}d\sigma \right) r_k'(s) ds
\end{align*}

\myindent The bracket term is equal to zero as $r_k$ is zero at $T_N$ for all $k$. Moreover we may infer from the construction of $r_k$ that 

\begin{equation}
    \int_{\R_+} |r_k'(s)|ds \leq C\beta_k
\end{equation}

with $C$ a universal constant independent of $k$ (indeed, we use that $r_k$ is a constant except maybe on a finite number of interval of length $2$ where its derivative is bounded by $c\beta_k\|\phi'\|_{\infty}$).

\myindent Finally there holds 

\begin{equation}
\left|\int_s^ta_n(\sigma)e^{i\omega \sigma}d\sigma \right| \leq C
\end{equation}

with $C$ a universal constant independent of $s,n,t,\omega$. Indeed, for any $n$, using Remark \eqref{behavioura}, we know that, on the one hand, $|a_n| \leq 1$ on $\R_+$, and on the other hand that, outside of a fixed finite number of intervals of length $2$ (yielding a bounded contribution to the integral), $a_n$ is either equal to zero, or equal to a finite linear combination with bounded number of terms of oscillating exponentials $e^{ift}$, with frequency $f = C'\beta_l$ with $C'$ a universal constant and $l\geq 0$. Thus, up to choosing $|m_0|$ larger, we can impose that there always holds $|f| < 1/2$. Hence, we are left with integrating oscillating exponentials $e^{i(f + \omega)\sigma}$ where $|f + \omega| \geq 1/2$ (since $|\omega| > 1$). A simple integration is enough to conclude the proof of the claim.

\myindent This yields the bound 

\begin{equation}
    J(k,n,\omega,t) \leq C\beta_k
\end{equation}

where $C$ is a universal constant.

\myindent Moreover we see that $r_k(s) = 0$ for all $s \geq T_{k+1}$, thus there holds 

\begin{equation}
    J(k,n,\omega,t) = 0 \quad \forall t \geq T_{k+1}
\end{equation}

\myindent From this we may infer the bound

\begin{equation}
    \alpha(t) \leq C\sum_{k \geq k(t)} \beta_k
\end{equation}

where we set $k(t)$ the smallest non-negative integer such that $t\leq T_{k+1}$. Using moreover the fast decay of $\beta_k$ we may further bound, up to taking a larger $C$ 

\begin{equation}
    \alpha(t) \leq C\beta_{k(t)}.
\end{equation}

\subsection{Estimates on $\beta(t)$}

\myindent As for $\beta(t)$, we see that for all $t$ there is a unique $l(t)$ such that $r_k(t) = 0$ as soon as $k\neq l(t)$, thus we find that

\begin{equation}
    \beta(t) = 4 r_{l(t)}(t).
\end{equation}

\myindent This yields the bound 

\begin{align*}
    \int_t^s \beta(\sigma) d\sigma &\leq 4\int_0^s r_{l(\sigma)}(\sigma) d\sigma \\
    &\leq C (k(s) +1) ;
\end{align*}

indeed, we see that the integral of $r_k$ over $\R_+$ is a constant independent of $k$.

\subsection{Conclusion of the estimates on $c^N$}

\myindent We may thus bound for $t\leq T_N$

\begin{equation}
    \|(c_n^N(t))\|_{l^1} \leq C\left(\beta_{k(t)} + \int_t^{T_N}\beta_{k(s)} \beta_{l(s)} \exp(C(k(s) + 1)) ds\right) 
\end{equation}

\myindent Now, from the construction of $(r_k(s))$ the holds $l(s) \geq k(s)$ thus $\beta_{l(s)} \leq \beta_{k(s)}$. Therefore there holds (for $t\leq T_N$)

\begin{equation}
    \|(c_n^N(t))\|_{l^1} \leq C\left(\beta_{k(t)} + \int_t^{T_N}\beta_{k(s)}^2 \exp(C(k(s) + 1)) ds\right) 
\end{equation}

\myindent Now, $k(s)$ is equal to $k$ on an intervall with measure $l_k$ such that $l_k\beta_k$ is equal to a constant, yielding the bound

    \begin{equation}
    \|(c_n^N(t))\|_{l^1} \leq C\left(\beta_{k(t)} + \sum_{k\geq k(t)}\beta_k e^{Ck} \right) 
\end{equation}

\myindent As $\beta_k$ is decaying faster than a double exponential there stands finally 
    \begin{equation}
    \forall t \leq T_N \quad \|(c_n^N(t))\|_{l^1} \leq C\beta_{k(t)} e^{Ck(t)}
\end{equation}

\section{Cauchy sequence and conclusion}

\subsection{Cauchy sequence}

\myindent We now prove that $(c^N)$ is a Cauchy sequence in $l^1(\Z)$. Set $M > N$ ; we look at the equation satisfied by $c^M - c^N$  

\begin{align*}
    (c_n^M - c_n^N)(t) &= -\sum_{m\in \Z^2} \int_t^{T_N} (c_m^M - c_m^N)(s) v_{n-m}(s) \left(e^{i\omega_{m,n}^+s} - e^{i\omega_{m,n}^-s}\right) ds + c_n^M(T_N) \\
    &- \sum_{m\notin \Gamma_{res}^+(n)} \int_{T_N}^{T_M} a_m(s)v_{n-m}(s)e^{-i\omega_{m,n}^+s}ds \ + \sum_{m\notin \Gamma_{res}^-(n)}\int_{T_N}^{T_M}a_m(s)v_{n-m}(s)e^{-i\omega_{m,n}^-s} ds
\end{align*}

\myindent Thus

\begin{align*}
    \|((c_n^M - c_n^N)(t))\|_{l^1} &\leq 2\int_t^{T_N} \|((c_n^M - c_n^N)(s))\|_{l^1} \|(v_n(s))\|_{l^1} ds + \|(c_n^M(T_N))\|_{l^1} \\
    &+ \sum_{n} \sum_{m\notin \Gamma_{res}^+(n)} \left|\int_{T_N}^{T_M} a_m(s)v_{n-m}(s)e^{-i\omega_{m,n}^+s}ds\right| \ + \sum_{n} \sum_{m\notin \Gamma_{res}^-(n)} \left|\int_{T_N}^{T_M} a_m(s)v_{n-m}(s)e^{-i\omega_{m,n}^-s}ds\right| \\
    &\leq 2\int_t^{T_N} \|((c_n^M - c_n^N)(s))\|_{l^1} \|(v_n(s))\|_{l^1} ds + C\beta_{k(T_N)} e^{Ck(T_N)} + C\beta_{k(T_N)} \\
    &\leq 2\int_t^{T_N} \|((c_n^M - c_n^N)(s))\|_{l^1} \|(v_n(s))\|_{l^1} ds + C\beta_{N - 1} e^{C(N-1)}
\end{align*}

\myindent Using backward Gronwall's lemma

\begin{equation}
    \|((c_n^M - c_n^N)(t))\|_{l^1} \leq C\beta_{N- 1}e^{C(N-1)}\left(1 + \int_t^{T_N} \beta(s) \exp\left(\int_t^s \beta(\sigma) d\sigma\right) ds\right)
\end{equation}

where $\beta(s) = 2\|(v_n(s))\|_{l^1}$. We know that $\beta(s) = 4 r_{l(s)}(s)$ and thus $\int_s^t \beta(\sigma) d\sigma \leq C(k(s) + 1)$. There holds

\begin{equation}
    \|((c_n^M - c_n^N)(t))\|_{l^1} \leq C\beta_{N- 1}e^{C(N-1)}\left(1 + \int_t^{T_N} \beta_{k(s)} \exp(Ck(s))ds\right)
\end{equation}

\myindent This upper bounds decays to zero as $N,M\to \infty$ if we fix $t$. This shows that $(c^N(t))$ is a Cauchy sequence in $l^1(\Z^2)$ and it thus converges to a $c(t)$ such that, using integral form of the differential equation, $b = a + c$ is a solution to $(\mathcal{RFS})$. There holds moreover

\begin{equation}
    \|(c_n(t))\|_{l^1} \leq C\beta_{k(t)} e^{Ck(t)}
\end{equation}

and this upper bound decays to zero as $t\to +\infty$ as expected.

\subsection{Growth of the Sobolev norm : qualitative result}

\myindent In order to conclude, we recall that $\|(a_n(t))\|_{l^2}$ is preserved and that for all $t\geq 0$ there are at most 5 of the $a_n$ that are nonzero. Therefore, we have on the one hand that for all $t \geq T_n$, $a_k = 0$ for $|k|< |m_n|$ and on the other hand that there exists $|k| \geq |m_n|$ such that $|a_k(t)| \geq \eps$ where $\eps > 0$ is a universal constant. Now, if we set $N$ large enough, we can ensure that $\|(c_n(t))\|_{l^1} \leq \eps /2$ for $t\geq T_N$. Therefore, for all $t \geq T_N$ with $N$ large enough, there exists $|k| \geq |m_N|$ such that $b_k = a_k + c_k$ satisfies $|b_k| \geq \eps / 2$. Now, this ensures that 

\begin{equation}
    \forall t \geq T_N \quad \|(b_n(t))\|_{H^s} \geq |k|^s |b_k| \geq \eps /2|m_N|^s 
\end{equation}

\myindent This already yields a qualitative result for theorem \eqref{result} as we already proved in section 3.4 that the potential $V$ along with all its time derivatives are decaying in all Sobolev norms when $t\to+\infty$

\subsection{Quantitative estimates on the growth rate}

\myindent We now investigate the quantitative bounds that we can hope to get on the rate of growth.

\myindent We first see that $T_n \leq C\beta_n^{-1}$ using the fast decay of $\beta_n$. Moreover as $|l_n| \leq C^n n!$  we find that

\begin{equation}
    T_n \leq \exp\left(C^n n! \log \left(C^n n!\right)\right)
\end{equation}

\myindent This yields the lower bound

\begin{equation}
    \|(b_n(t))\|_{H^s} \geq \delta |m_{n(t)}|^s
\end{equation}

where $\delta > 0$ is a constant, and $n(t)$  the largest integer $n$ such that $\exp\left(C^n n! \log \left(C^n n!\right)\right) \leq t$. Now, we know moreover that $|m_n| \geq c (n-1)!$, thus leading to the lower bound 

\begin{equation}
    \|u(t)\|_{H^s} \geq \eps c^s ((n(t) - 1)!)^s.
\end{equation}

\myindent In order to obtain better bounds, take $\eta > 0$. We first use Stirling's formula 

\begin{equation}
    n! \sim \left(\frac{n}{e}\right)^n \sqrt{2\pi n}
\end{equation}

which ensures that provided $n$ is large enough

\begin{equation}
    C^n n! \log \left(C^n n!\right) \leq ((1+\eta)n)^{(1 + \eta)n}
\end{equation}

\myindent Now set $f(x) := x^x$. We find that provided 

\begin{equation}
    f((1 + \eta)n) \leq \log(t)
\end{equation}

then, provided $n$ is large enough, there stands that $n\leq n(t)$. Now, provided $n$ is large enough, there also holds 

\begin{equation}
     (n - 1)! \geq ((1-\eta)n)^{(1-\eta)n} = f((1-\eta)n).
\end{equation}

\myindent Thus, setting $E(x)$ the largest integer $k$ such that $k\leq x$ we can find a lower bound of the form 

\begin{align*}
    \|u(t)\|_{H^s} &\geq \left(c f\left(\frac{1 - \eta}{1 + \eta} E(f^{-1}(\log(t)))\right) \right)^s \\
    &\geq c_{s,\eta}\exp\left(s\frac{(1 - \eta)^2}{1 + \eta} f^{-1}(\log(t)) \log\left(\frac{(1 - \eta)^2}{1 + \eta} f^{-1}(\log(t)) \right)\right) \ \text{provided t is large enough}\\
    &\geq c_{s,\eta}\exp\left(s\frac{(1 - \eta)^3}{1 + \eta} f^{-1}(\log(t)) \log\left(f^{-1}(\log(t))\right)\right) \ \text{provided t is large enough} \\
    &\geq c_{s,\eta} \left(\log(t)\right)^{\frac{s(1 - \eta)^3}{1 + \eta}}.
\end{align*}

\myindent As we may choose $\eta$ arbitrarily, we find that given any $\delta,s > 0$ there exists $c_{\delta,s}> 0$ such that for $t > 1$

\begin{equation}
    \|u(t)\|_{H^s} \geq c_{\delta,s} (\log t)^{s(1 - \delta)},
\end{equation}

thus concluding the proof of theorem \eqref{result}

\subsection{Estimates on the decay rate of $V$}

\myindent We now prove similar upper bounds on the decay rate of the potential $V(t)$. Fix $s \geq 0$ and $m\in \N\cup\{0\}$. Thanks to equation \eqref{decay}, we may bound 

\begin{equation}
    \|\partial^m_t V(t,\cdot)\|_{H^s} \leq c |l_{k(t)}|^{M-|l_{k(t)}|}
\end{equation}

where $M = M_{m,s} > 0$ and $k(t)$ is the unique $k\geq 0$ such that $r_{k(t)} \neq 0$. We may furthemore infer from previous subsection that given $\delta > 0$ there exists $c_{\delta} > 0$ such that 

\begin{equation}
    |l_{k(t)}| \geq c_{\delta} (\log(t))^{1 - \delta}
\end{equation}

\myindent Thus 

\begin{equation}
     \|\partial^m_t V(t,\cdot)\|_{H^s} \leq C_{\delta} \exp((M_{m,s} - (\log(t))^{1 - \delta}) (1-\delta)\log(\log(t))) 
\end{equation}

\myindent As this holds for all $\delta > 0$, we may conclude that for all $\delta > 0$ there exists $C_{\delta,m,s}$ such that 

\begin{equation}\label{decayV}
     \|\partial^m_t V(t,\cdot)\|_{H^s} \leq C_{\delta,m,s} \exp( - (\log(t))^{1 - \delta}) \log(\log(t))) 
\end{equation}

\myindent As this yields a quantitative bounds for the decay of $V$, it should be noted that it is subpolynomial in the sense that the upper bound decays slower than $t^{-\eps}$ for all $\eps > 0$. It doesn't seem that we can improve the bound, as on $[T_N, T_{N+1}]$ $\|V(t)\|_{H^1}$ is of order $\beta_{N}$, and $T_{N+1}$ is of order $\beta_{N+1}^{-1}$. As for all $\eps > 0$ asymptotically there holds $\beta_{N+1}^{\eps} << \beta_N$ we thus cannot hope for a better bound. 

\printbibliography

\end{document}